\documentclass[11pt,a4paper]{article} 

\usepackage{amsfonts, amsmath, wasysym}
\usepackage{stmaryrd}

\usepackage{tikz}
\usetikzlibrary{shapes, positioning, patterns}

\usepackage{amssymb,amsthm,
paralist
}

\usepackage{
latexsym,
}

\usepackage{url}

\definecolor{darkgreen}{rgb}{0,0.5,0}
\definecolor{darkred}{rgb}{0.7,0,0}
\usepackage[colorlinks, 
citecolor=darkgreen, linkcolor=darkred
]{hyperref}

\usepackage[a4paper, margin=1in]{geometry}

\theoremstyle{plain}

\numberwithin{equation}{section}

\newcommand{\beq}{\begin{equation}}
\newcommand{\beql}[1]{\begin{equation}\label{#1}}
\newcommand{\eeq}{\end{equation}}
\newcommand{\beqa}{\begin{equation}\begin{aligned}}
\newcommand{\eeqa}{\end{aligned}\end{equation}}
\newcommand{\brmk}{\begin{rmk}}
\newcommand{\ermk}{\end{rmk}}
\newcommand{\partref}[1]{\hbox{(\csname @roman\endcsname{\ref{#1}})}}

\newcommand{\Ric}{{\mathrm{Ric}}}

 \newtheorem{thm}{Theorem}[section]

\newtheorem{lem}[thm]{Lemma}
\newtheorem{prop}[thm]{Proposition}

\newtheorem{rmk}[thm]{Remark}

\def\GG{\mathrm{G}}

\renewcommand{\i}{\imath}
\renewcommand{\j}{\jmath}
\def\iRic{\i(\textnormal{Ric}_{0})}

\usepackage{varwidth}

\makeatletter
\renewcommand{\@makefntext}[1]{\noindent\makebox[0pt][l]{\@makefnmark}\hspace*{0.6em}#1}
\makeatother

\title{\Large Properties and Perturbations of Extremally Ricci-Pinched $G_{2}$-Structures\textsuperscript{*}}
\author{Aaron Kennon}
\date{24 September 2026}

\begin{document}

\maketitle

\begingroup
\renewcommand{\thefootnote}{*}
\footnotetext{The author is supported by the NSF via DMS-2342135.\par\noindent MSC 2020: 53E20, 53C20.}
\endgroup

\begin{abstract}
We study properties of Extremally Ricci-Pinched (ERP) $G_2$-Structures on compact $7$-manifolds. We provide characterizations of the ERP condition in terms of the traceless $\ast$-Ricci tensor, the Hodge Laplacian of the torsion, the Ricci eigenvalues, the Laplacian flow, and the $27$-dimensional Weyl curvature component. We also prove that the automorphism group of a compact ERP $G_2$-Structure is finite. We then identify perturbations of ERP $G_2$-Structures, both in general and for specific examples, that preserve the ERP condition.
\end{abstract}

\section{Introduction}

A primary goal of the Laplacian flow of $\GG_{2}$-Structures is to help identify conditions under which a $G_2$-Structure with torsion may be deformed to one which is torsion-free. Although the Laplacian flow in principle could be applied to any type of $G_2$-Structure, it is best motivated in the context of closed $\GG_{2}$-Structures since the closed condition is manifestly preserved under the flow and the flow is known to have short-time existence and uniqueness starting at closed initial data \cite{BryantXu2004}. More recently, long-time existence results have been established for the Laplacian flow of closed $\GG_{2}$-Structures, as well as dynamical stability and compactness results \cite{LotayWei2017, LotayWei2019}. \\

Within the class of closed $\GG_{2}$-Structures, the Extremally Ricci-Pinched (ERP) subclass constitute precisely how pinched the Ricci tensor of a closed $\GG_2$-Structure may be relative to its scalar curvature without it necessarily being torsion-free \cite{Bryant2005}. Besides their significance from the Riemannian perspective, the torsion two-form, which measures the deviation of a closed $\GG_{2}$-Structure from being torsion-free, also obeys several special identities for ERP $\GG_{2}$-Structures. These structures are also significant from the perturbative perspective as the Laplacian flow preserves this ERP condition \cite{FinoRaffero2021}. As a result, by forwards and backwards uniqueness, a closed $\GG_{2}$-Structure on a compact manifold cannot become ERP in finite time. In this paper we study the class of ERP $G_2$-Structures under the premise that understanding them may provide more insight into the class of closed $G_2$-Structures more broadly and that their behavior may be useful for gaining insight into when the Laplacian flow may exist for all time and converge to a torsion-free $\GG_2$-Structure.   \\

We begin by introducing background on $\GG_2$-Structures in \S2. In \S3 we give characterizations of the ERP condition using the traceless $\ast$-Ricci tensor, the Hodge Laplacian of the torsion, a pointwise bound on the Ricci eigenvalues, the Laplacian flow, and a pointwise bound on the $27$-dimensional Weyl curvature component. We also prove that the automorphism group is finite. In \S4 we use the data of an ERP $\GG_2$-Structure to define natural perturbations preserving the ERP condition. In \S5 we study perturbations of three examples in the classification of Lauret and Nicolini.

\section{Background on $\GG_{2}$-Structures}

Here we include a general introduction to $\GG_{2}$-Structures on seven-manifolds. The material here may be found in the review by Bryant \cite{Bryant2005}, the review by Karigiannis \cite{Karigiannis2010}, or the monograph by Joyce \cite{JoyceBook}. \\

A \textit{$\GG_{2}$-Structure} is a principle subbundle of the frame bundle with structure group $\GG_{2}$. The existence of such a structure without any further qualifications is just a topological matter, namely a seven-manifold admits $\GG_{2}$-Structures if and only if it is orientable and spin. We may also characterize the existence of a $\GG_{2}$-Structure in terms of a three-form $\varphi$ pointwise identified with a model three-form $\varphi_{0}$ on $\mathbb{R}^{7}$. 

\begin{equation}\label{G2pointwise}
\varphi_0 = dx^{123} + dx^{145} + dx^{167} + dx^{246} - dx^{257} -
dx^{347} - dx^{356}
\end{equation}

Such a three-form is called \textit{positive} and we also refer to such a form as the \textit{G$_2$ three-form}. A positive three-form algebraically defines a Riemannian metric $g$ and also thereby determines a four-form $\psi$ by taking the Hodge star of the three-form with respect to this metric. The metric is given in terms of the three-form $\varphi$ as

 \begin{equation}\label{G2metric}
X\lrcorner\varphi\wedge Y\lrcorner\varphi \wedge \varphi = 6g(X,Y)\textnormal{vol}
 \end{equation}

On a manifold with a positive three-form $\varphi$, we may decompose the spaces of differential k-forms into irreducible $\GG_{2}$ representations. The equations defining the various representations are then all given in terms of $\varphi$ and its Hodge dual $\psi$ \cite{Karigiannis2010}.

\begin{align}
 \hspace{-100mm} & \Lambda^{1}(T^{*}M) = \Lambda^{1}_{8}(T^{*}M) \\
& \Lambda^{2}(T^{*}M) = \Lambda^{2}_{7}(T^{*}M) \oplus \Lambda^{2}_{14}(T^{*}M) \\
& \Lambda^{3}(T^{*}M) = \Lambda^{3}_{1}(T^{*}M) \oplus \Lambda^{3}_{7}(T^{*}M) \oplus \Lambda^{3}_{27}(T^{*}M) 
\end{align}

\begin{align}\label{G2formdecompositions}
& \Lambda^{2}_{7}(T^{*}M) = \{ \beta \in \Lambda^{2}(T^{*}M) : *(\varphi \wedge \beta) = 2\beta \} \\
& \Lambda^{2}_{14}(T^{*}M) =  \{ \beta \in \Lambda^{2}(T^{*}M) : *(\varphi \wedge \beta) = -\beta \} \\ 
& \Lambda^{3}_{1}(T^{*}M)= \{  f\varphi \hspace{2mm} \textnormal{for} \hspace{1mm} f \in C^{\infty}(M) \} \\
& \Lambda^{3}_{7}(T^{*}M) =  \{ X \lrcorner \psi \hspace{2mm} \textnormal{for}\hspace{1mm} X \in \Gamma(M) \} \\
& \Lambda^{3}_{27}(T^{*}M) =  \{ \eta \in \Lambda^{3}(T^{*}M): \eta \wedge \varphi = \eta \wedge \psi=0 \} 
\end{align}

The spaces $S_{0}^{2}(M)$ and $\Lambda_{27}^{3}(T^{\ast}M)$ are isomorphic as $G_2$ representation spaces and a particular isomorphism $\i_{\varphi}$ between them may be given by extending by linearity the following map on decomposable symmetric traceless two-tensors $\alpha \circ \beta$ \cite{Bryant2005}. 

\begin{equation}
\i_{\varphi}(\alpha\circ\beta) = 
\alpha\wedge\ast(\beta\wedge\ast\varphi)
+\beta\wedge\ast(\alpha\wedge\ast\varphi)
\end{equation}

In local coordinates the map $\i_{\varphi}$ acting on a symmetric two-tensor $h$ takes the form 

\begin{align} \label{ilocal}
\i_{\varphi}(h) &= h_{il}\varphi_{ljk} \hspace{1mm} dx^{i}\wedge dx^{j}\wedge dx^{k}\\
&= \frac{1}{3}(h_{il}\varphi_{ljk} -h_{jl}\varphi_{lik} - h_{kl}\varphi_{lji}) \hspace{1mm} dx^{i}\wedge dx^{j}\wedge dx^{k} \nonumber
\end{align}

With this normalization,
\begin{equation}\label{eq:i-norm}
  \i_\varphi(g_\varphi)=6\varphi
  \qquad |\i_\varphi(h)|^2=8|h|^2
  \quad\text{for }h\in S^2_0(M)
\end{equation}

The inverse of this isomorphism, denoted $\j_{\varphi}$, as a map from $\Lambda_{27}^{3}$ to $S_{0}^{2}(M)$ takes the form
 
\begin{equation}
\j_{\varphi}(\gamma)(v,w) = \ast\bigl((v\lrcorner\varphi)\wedge(w\lrcorner\varphi)\wedge\gamma\bigr)
\end{equation}

Following Fern\'andez and Gray \cite{FernandezGray1982}, for an arbitrary $\GG_{2}$-Structure we may decompose the exterior derivatives of the $\GG_{2}$ three-form and four-form into irreducible $\GG_{2}$ representations.

\begin{align}\label{fernandezgray}
& d\varphi = \tau_{0}\psi + 3 \tau_{1} \wedge \phi + \ast \tau_{3} \\
& d\psi = 4 \tau_{1} \wedge \psi - \ast \tau_{2}
\end{align}

This decomposition identifies special functions and differential forms $\tau_{0} \in C^{\infty}(M)$, $\tau_{1} \in \Gamma(M)$, $\tau_{2} \in \Lambda^{2}_{14}(T^{*}M)$, and $\tau_{3}\in \Lambda^{3}_{27}(T^{*}M)$ called \textit{torsion forms} of the $G_2$-Structure. There are then sixteen possibilities for types of $\GG_{2}$-Structures depending on which of these torsion forms vanish. For a positive three-form underlying a $\GG_{2}$-holonomy metric all of the torsion forms must vanish. The three-form underlying a metric with holonomy contained in $\GG_{2}$ is then said to be \textit{torsion-free}.\\

For a closed $\GG_{2}$-Structure, the Fern\'andez--Gray decomposition implies the only non-vanishing torsion form must be $\tau_2$, which from now on we will denote as $\tau$ since from now on we will only be considering closed $G_2$-Structures. Any form in $\Lambda^{2}_{14}(T^{*}M)$ including $\tau$ obeys the following properties \cite{Bryant2005}. \footnote{In this paper unless stated otherwise we use $|\cdot|$ to denote the differential form norm and $\|\cdot\|$ to denote the tensor norm, which are related by a factor of $k!$}

\begin{align}\label{14dimrepproperties}
&|\tau\wedge\tau|^2 = |\tau|^4\\ 
&\tau\wedge\ast(\tau\wedge\tau) = |\tau|^2 \ast\tau + \frac{1}{3}\ast(\tau\wedge\tau\wedge\tau)\wedge\psi
\end{align}

In local coordinates since $\tau\in\Lambda^{2}_{14}(T^{*}M)$ it satisfies

\begin{equation}
\tau_{ij}\psi_{ijab} = -2\tau_{ab}
\end{equation}

We may also relate $\tau$ to the torsion tensor of the $G_2$-Structure

\begin{equation}
T_{ab} = \frac{1}{24}\Big(\nabla_{a}\varphi_{ijk}\Big)\psi_{bijk}
\end{equation}

For a closed $G_2$-Structure since $\tau$ is the only non-vanishing torsion form \cite{Karigiannis2010}

\begin{equation}
T_{ab} = -\frac{1}{2}\tau_{ab}
\end{equation}

The torsion tensor is used to characterize the covariant derivative of $\varphi$ in the following manner

\begin{equation}\label{G2TorsionTensor}
\nabla_{m}\varphi_{ijk} = T_{mn}\psi_{nijk}
\end{equation}

For such a structure, the Ricci and scalar curvatures also have relatively simple dependence on the torsion $\tau$ \cite{Bryant2005}. 

\begin{align}\label{closedRicciScalar}
\textnormal{Ric}(g_\varphi) &= 
  \biggl(\frac{1}{4}\,|\tau|^2\biggl)\,g_\varphi - \j_{\varphi}\Bigl( \frac{1}{4}\, d\tau - \frac{1}{8}\,\ast(\tau\wedge\tau) \Bigr) \\
\textnormal{Scal}(g_\varphi) &= -\frac{1}{2}|\tau|^2
\end{align}

Using the curvature convention $\textnormal{Ric}_{ij}=R_{kijk}$, following Cleyton--Ivanov \cite{CleytonIvanov2007,CleytonIvanov2008} we define the $\ast$-Ricci tensor by
\begin{equation}\label{eq:star-ricci-definition}
  \textnormal{Ric}^{\ast}_{ij}
  = R_{abcd}\,\varphi_{abi}\,\varphi_{cdj}
\end{equation}
Its trace satisfies
$\textnormal{tr}_{g_\varphi}\textnormal{Ric}^{\ast}
=-2\,\textnormal{Scal}(g_\varphi)$ and therefore its traceless part is
\begin{equation}
  \textnormal{Ric}^{\ast}_{0}
  = \textnormal{Ric}^{\ast}
    + \frac{2}{7}\textnormal{Scal}(g_\varphi)\,g_\varphi
\end{equation}

In local coordinates, the Ricci curvature may also be given as \cite{LotayWei2017}

\begin{equation} \label{closedG2Ricci}
\Ric_{ik} = (\nabla_{j}T_{li})\varphi_{kjl} - T_{ij}T_{jk}
\end{equation}

In particular, the scalar curvature is non-positive and vanishes identically if and only if the $\GG_{2}$-Structure is torsion-free. In this case the metric is not only scalar-flat but also Ricci-flat. An expression for the exterior derivative of $\tau$ follows by isolating the trace-free part of the Ricci curvature $\textnormal{Ric}_0$ and applying the isomorphism $\i_{\varphi}$ to \ref{closedRicciScalar}.

\begin{equation}\label{dtau}
d\tau = \frac{3}{14}|\tau|^2\varphi + \frac{1}{2}\ast(\tau\wedge\tau) - \frac{1}{2} \i_{\varphi} (\textnormal{Ric}_{0})
\end{equation}

Similarly, the $\ast$-Ricci curvature formula of Cleyton and Ivanov may be rearranged to 
\cite{CleytonIvanov2008} give
\begin{equation}\label{eq:star-ricci-torsion}
  d\tau
  = \frac{1}{8}\i_\varphi(\textnormal{Ric}^{\ast}_{0})
    + \frac{3}{28}|\tau|^2\varphi
    - \frac{1}{4}\ast(\tau\wedge\tau)
\end{equation}

As far as the Riemannian curvature of closed $\GG_2$-Structures is concerned, Bryant showed that all closed $\GG_2$-Structures that are not torsion-free satisfy the following inequality \cite{Bryant2005}. 

\begin{equation}\label{pinchinginequality}
\int_{M} |\i_{\varphi} (\textnormal{Ric}_{0})|^2 \textnormal{vol} \geq \frac{8}{21} \int_{M} |\tau|^4 \textnormal{vol}
\end{equation}

Closed $\GG_2$-Structures which saturate this inequality are called \textit{Extremally Ricci-Pinched (ERP)}. This ERP condition may also be characterized in terms of a pointwise condition on the torsion two-form $\tau$.

\begin{equation}\label{erppointwise}
d\tau = \frac{1}{6}|\tau|^2\varphi + \frac{1}{6}\ast(\tau\wedge\tau)
\end{equation}

This pointwise condition may also be cast as an expression for the trace-free Ricci tensor.

\begin{equation}\label{eq:erp-tracefree-ricci}
\i_\varphi(\textnormal{Ric}_0) = \frac{2}{3} (\ast(\tau\wedge\tau) + \frac{1}{7}|\tau|^2\varphi)
\end{equation}

Following the notation of Cleyton and Ivanov \cite{CleytonIvanov2008}, let $\bar\nabla$ denote the canonical $\GG_2$ connection obtained by projecting the Levi-Civita connection onto $\mathfrak g_2$. We introduce the notation
\begin{equation}\label{eq:canonical-torsion-derivative}
  d^{\bar\nabla}\tau
  =d\tau-\frac{1}{6}|\tau|^2\varphi
    -\frac{1}{6}\ast(\tau\wedge\tau)
  \in\Lambda^3_{27}(T^*M)
\end{equation}
Note that $\varphi$ is ERP if and only if $d^{\bar\nabla}\tau=0$. We also have
\begin{equation}\label{eq:torsion-square-projection}
  \pi_{27}\bigl(\ast(\tau\wedge\tau)\bigr)
  =\ast(\tau\wedge\tau)+\frac{1}{7}|\tau|^2\varphi
  \qquad
  \left|\pi_{27}\bigl(\ast(\tau\wedge\tau)\bigr)\right|^2
  =\frac{6}{7}|\tau|^4
\end{equation}

The Riemannian curvature meanwhile has the orthogonal $\GG_2$ decomposition \cite{CleytonIvanov2008}
\begin{equation}\label{eq:g2-curvature-decomposition}
  \textnormal{Rm}=S+R_0+W_{27}+W_{64}+W_{77},
\end{equation}
where $S$ and $R_0$ are determined by the scalar and trace-free Ricci curvatures, respectively, and the subscripts on the Weyl components indicate their dimensions. The component $W_{27}$ is determined by the symmetric trace-free tensor
\begin{equation}\label{eq:weyl-ricci-definition}
  \textnormal{Ric}^{W}
  =\frac{1}{20}\left(4\textnormal{Ric}_0
    -5\textnormal{Ric}^{\ast}_0\right)
  \qquad
  \|W_{27}\|^2=\frac{15}{28}|\textnormal{Ric}^{W}|^2
\end{equation}
For ERP $\GG_2$-Structures, the Weyl component $W_{27}$ satisfies $\|W_{27}\|^2=\frac{9}{320}|\tau|^4$ and hence is nonzero wherever $\tau\neq0$.\\

In the compact setting Bryant showed these structures have some additional rather remarkable properties.

\medskip
\begin{prop} [Properties of ERP $\GG_2$-Structures \cite{Bryant2005}] \label{thm: BryantERPResults} Let $\varphi$ be an ERP $\GG_2$-Structure with corresponding torsion two-form $\tau$ on a compact 7-manifold. Then

\begin{align}\label{erpproperties}
& d|\tau|^2 = 0 \\
& \tau\wedge\tau\wedge\tau=0 \\
& \Delta_{\varphi} (\tau\wedge\tau) = 0 
\end{align}

\end{prop}

The first property means that the scalar curvature of such a structure is a negative constant. Since $\tau\wedge\tau$ cannot vanish for representation-theoretic reasons \cite{Bryant2005}, but $\tau\wedge\tau\wedge\tau$ necessarily does, that implies that the four-form $\tau\wedge\tau$ and its Hodge dual $\ast(\tau\wedge\tau)$ are simple. This implies an orthogonal splitting of the tangent bundle into integrable subbundles $TM = P \oplus Q$, where $P$ is calibrated by $-|\tau|^2 \ast(\tau\wedge\tau)$ and $Q$ by $-|\tau|^2(\tau\wedge\tau)$. Given the expression for the Ricci tensor, this splitting further implies that the Ricci curvature is non-positive with eigenvalues $-\frac{1}{6}|\tau|^2$ of multiplicity three and 0 of multiplicity four. \\

In terms of this orthogonal splitting, we may express the full Ricci tensor as

\begin{equation}\label{ERPricci}
\Ric(g_{\varphi}) = -\frac{1}{6}|\tau|^2 g_{\varphi}|_{P}
\end{equation}

The original example of an ERP $G_2$-Structure was constructed by Bryant \cite{Bryant2005} and an example admitting a compact quotient amongst other examples were constructed by Kath and Lauret \cite{KathLauret2020}. Additional examples in the non-compact setting are known to be soliton solutions of the Laplacian flow \cite{Ball2019}, however, in the compact setting the three-form underlying an ERP $G_2$-Structure cannot be exact and therefore it cannot be an expanding soliton. \cite{Kennon2024}.  \\

\section{Characterizations of ERP Condition}

We first characterize the ERP condition by adapting Bryant's pinching argument to the traceless $\ast$-Ricci tensor.

\begin{prop} [$\ast$-Ricci Characterization of ERP] \label{thm:star-ricci-erp} Let $\varphi$ be a closed $\GG_2$-Structure on a compact $7$-manifold $M$ without boundary. Suppose that, for some constant $0\leq C\leq1$,
\begin{equation}\label{eq:star-ricci-pinching}
  \|\textnormal{Ric}^{\ast}_{0}\|^2
  \leq \frac{100}{21}C\,\textnormal{Scal}(g_\varphi)^2
\end{equation}
holds pointwise. If $C<1$, then $\varphi$ is torsion-free. If $C=1$, then $\varphi$ satisfies the ERP condition \ref{erppointwise}. Moreover, every ERP $\GG_2$-Structure satisfies
\begin{equation}\label{eq:erp-star-ricci}
  \textnormal{Ric}^{\ast}_{0}=5\,\textnormal{Ric}_0
  \qquad
  \|\textnormal{Ric}^{\ast}_{0}\|^2
  =\frac{100}{21}\textnormal{Scal}(g_\varphi)^2
\end{equation}
\end{prop}

\textit{Proof.} Using equation \ref{eq:torsion-square-projection} after wedging \ref{eq:star-ricci-torsion} with $\tau\wedge\tau$ yields
\begin{equation}\label{eq:star-ricci-divergence}
  d\left(\frac{1}{3}\tau\wedge\tau\wedge\tau\right)
  =\left[\frac{1}{8}\left\langle
      \i_\varphi(\textnormal{Ric}^{\ast}_{0}),\pi_{27}(\ast(\tau\wedge\tau))
    \right\rangle-\frac{5}{14}|\tau|^4\right]\textnormal{vol}
\end{equation}
By Cauchy--Schwarz, 
\begin{align*}
  \ast d\left(\frac{1}{3}\tau\wedge\tau\wedge\tau\right)
  &\leq \sqrt{\frac{3}{28}}\,
    |\textnormal{Ric}^{\ast}_{0}|\,|\tau|^2
    -\frac{5}{14}|\tau|^4\\
  &\leq\frac{5}{14}(\sqrt C-1)|\tau|^4
\end{align*}
If $C<1$, integration and Stokes' theorem imply $\tau=0$. If $C=1$, the left-hand side is non-positive and has zero integral, so equality holds pointwise in both inequalities. At each point where $\tau\neq0$, equality in Cauchy--Schwarz and the divergence identity above imply
\begin{equation}\label{eq:erp-star-ricci-form}
  \i_\varphi(\textnormal{Ric}^{\ast}_{0})=\frac{10}{3}\pi_{27}(\ast(\tau\wedge\tau))
\end{equation}
At points where $\tau=0$, the same identity holds trivially via the pointwise identity. Conversely, substituting \ref{erppointwise} into \ref{eq:star-ricci-torsion} gives the preceding expression for $\i_\varphi(\textnormal{Ric}^{\ast}_{0})$. Comparing this with \ref{eq:erp-tracefree-ricci} yields $\textnormal{Ric}^{\ast}_{0}=5\,\textnormal{Ric}_0$. Finally, the norm identity \ref{eq:i-norm} and the expression for the scalar curvature give
\[
  |\textnormal{Ric}^{\ast}_{0}|^2
  =\frac{100}{72}|\pi_{27}(\ast(\tau\wedge\tau))|^2
  =\frac{25}{21}|\tau|^4
  =\frac{100}{21}\textnormal{Scal}(g_\varphi)^2
\]
\qed

We next turn to a characterization in terms of the torsion. We begin with an integral identity relating $d\tau$ to the trace-free Ricci tensor.

\begin{lem} [Global Curvature Equality] \label{thm: L2CurvatureLemma} For a closed $\GG_2$-Structures $\varphi$ on a compact manifold $M$, its associated torsion $\tau$ and traceless Ricci tensor $\iRic$ obey the following global relationship.

\begin{equation}
\int_{M}|d\tau|^2\textnormal{vol} = \frac{1}{14} \int_{M} |\tau|^4 \textnormal{vol} + \frac{1}{4} \int_{M} |\iRic|^2 \textnormal{vol}
\end{equation}

\end{lem}

\textit{Proof.} We start with the general expression for $d\tau$ for a closed $\GG_2$-Structure.

\begin{equation}
d\tau = \frac{3}{14}|\tau|^2\varphi + \frac{1}{2}\ast(\tau\wedge\tau) - \frac{1}{2} \i (\textnormal{Ric}_{0})
\end{equation}

We first want to solve for the $L^2$ inner product of $\ast(\tau\wedge\tau$) and $\iRic$. We do so by wedging the expression for $d\tau$ against $\tau\wedge\tau$, integrating, and citing Stokes' theorem.

\begin{equation}\label{crossterm}
\langle\ast(\tau\wedge\tau), \iRic\rangle_{L^2} = \frac{4}{7} \int_{M} |\tau|^4 \textnormal{vol}
\end{equation}

We then compute the $L^2$ norm of $d\tau$ using the general expression for it in terms of $\tau$ and $\iRic$. The expression involves cross terms of the form $\langle\iRic, \ast(\tau\wedge\tau)\rangle_{L^2}$.

\begin{equation}
\int_{M}|d\tau|^2\textnormal{vol} = \frac{5}{14}\int_{M}|\tau|^4 \textnormal{vol} -\frac{1}{2} \int_{M} \iRic\wedge\tau\wedge\tau +\frac{1}{4}\int_{M} |\iRic|^2\textnormal{vol}
\end{equation}

We can then plug in the cross terms from \ref{crossterm} in terms of the scalar curvature and simplify to get the desired relation. \qed \\

Utilizing this result, it then follows that the ERP condition may be characterized within the broader class of closed $G_2$-Structures succinctly in terms of a condition on the Hodge Laplacian of its torsion.  

 \begin{prop} [ERP Characterization] \label{prop: ERPChar} Let $\varphi$ be a closed $G_2$-Structure on a compact 7-manifold $M$ with torsion two-form $\tau$. If $\tau$ pointwise satisfies $\Delta\tau = \frac{1}{6}|\tau|^2\tau$ then the closed $G_2$-Structure is ERP. 

 \end{prop}

 \textit{Proof}. Assuming the given constraint on $\Delta\tau$ we may simply integrate by parts against $\tau$ to get

 \begin{equation}\label{intparts}
|d\tau|^2_{L^2} = \langle \tau, \Delta\tau \rangle_{L^2} = \frac{1}{6} \int_{M} |\tau|^4 \hspace{1mm}\textnormal{vol}
\end{equation}

We may then plug this in for $d\tau$ in Lemma \ref{thm: L2CurvatureLemma} and rearrange to get the ERP condition in the global form \ref{pinchinginequality}. \qed

\begin{rmk}
Note that $\langle \tau, \Delta\tau\rangle_{L^2}$ cannot be smaller than this threshold value or else it would violate the ERP Pinching inequality. Applying the $L^2$ Cauchy-Schwarz inequality to $\langle \tau, \Delta\tau\rangle_{L^2}$ thus implies a pinching inequality for $|\Delta\tau|$ as well which is saturated by ERP $G_2$-Structures.
\end{rmk}

The ERP condition may also be characterized by a pointwise bound on the eigenvalues of the Ricci tensor. This characterization builds off of work of Payne, who shows closed $\GG_2$-Structures can never have negative Ricci curvature in the compact setting \cite{Payne2025}.

\begin{prop} [Ricci Eigenvalue Characterization of ERP] \label{thm: ERPRicciEigenvalues}
Let $\varphi$ be a closed $\GG_2$-Structure with torsion two-form $\tau$ on a $7$-manifold $M$ satisfying
\[
  \textnormal{Ric}(g_\varphi)\leq0
\]
Let $\lambda_{\min}$ denote the smallest eigenvalue of $\textnormal{Ric}(g_\varphi)$. If
\begin{equation}\label{eq:erp-ricci-eigenvalue-bound}
  \lambda_{\min}\geq\frac{1}{3}\textnormal{Scal}(g_\varphi)
\end{equation}
pointwise on $M$, then $\varphi$ is ERP.
\end{prop}

\textit{Proof.} Starting from Payne's result \cite{Payne2025} and the Ricci curvature assumption, we find
\[
  \bigl\langle d\tau,\ast(\tau\wedge\tau)\bigr\rangle=0
\]
Using \ref{eq:canonical-torsion-derivative} and \ref{14dimrepproperties}, we find
\[
  \left\langle d^{\bar\nabla}\tau,
    \pi_{27}\bigl(\ast(\tau\wedge\tau)\bigr)\right\rangle=0
\]
Combining \ref{dtau}, \ref{eq:canonical-torsion-derivative}, and \ref{eq:torsion-square-projection} yields
\[
  \i_\varphi(\textnormal{Ric}_0)
  =-2d^{\bar\nabla}\tau
    +\frac{2}{3}\pi_{27}\bigl(\ast(\tau\wedge\tau)\bigr)
\]
As these two terms are orthogonal, the norm identities \ref{eq:i-norm} and \ref{eq:torsion-square-projection} imply
\[
  |\textnormal{Ric}_0|^2
  =\frac12|d^{\bar\nabla}\tau|^2+\frac1{21}|\tau|^4
\]
Adding $\textnormal{Scal}(g_\varphi)^2/7$ and using the expression for the scalar curvature, we conclude that
\begin{equation}\label{eq:ricci-eigenvalue-erp-defect}
  |\textnormal{Ric}|^2
  =\frac{\textnormal{Scal}(g_\varphi)^2}{3}+\frac12|d^{\bar\nabla}\tau|^2
\end{equation}
On the other hand, every Ricci eigenvalue $\lambda_i$ lies in $[\textnormal{Scal}(g_\varphi)/3,0]$, so $\lambda_i^2\leq(\textnormal{Scal}(g_\varphi)/3)\lambda_i$. Summing over $i$ gives
\[
  |\textnormal{Ric}|^2=\sum_{i=1}^7\lambda_i^2
  \leq\frac{\textnormal{Scal}(g_\varphi)}{3}\sum_{i=1}^7\lambda_i
  =\frac{\textnormal{Scal}(g_\varphi)^2}{3}
\]
Thus the two expressions for $|\textnormal{Ric}|^2$ force $d^{\bar\nabla}\tau=0$ which is equivalent to the ERP condition. \qed

\medskip
The ERP condition is also known to be well-behaved under the Laplacian flow via the work of Fino and Raffero \cite{FinoRaffero2021}. We now prove a converse to their result.

\begin{prop} [Laplacian Flow Characterization of ERP] \label{thm: ERPFlowCharacterization} Let $\varphi$ be a closed $\GG_2$-Structure with constant scalar curvature and torsion two-form $\tau$ on a compact $7$-manifold $M$. Suppose that the Laplacian flow starting at $\varphi$ satisfies
\begin{equation}\label{eq:erp-flow-ansatz}
  \varphi(t)=\varphi+f(t)d\tau
  \qquad
  f(t)=\frac{6}{|\tau|^2}
    \left(\exp\left(\frac{|\tau|^2}{6}t\right)-1\right)
\end{equation}
for all sufficiently small $t\geq0$, where $|\tau|$ is computed using the initial structure. Then $\varphi$ is ERP,  $\varphi(t)$ remains ERP, and the Laplacian flow is eternal.
\end{prop}

\textit{Proof.} All norms, projections, and volume forms without an explicit time dependence are taken with respect to the initial structure $\varphi$. We then find
\begin{equation}\label{eq:erp-flow-projections}
  \pi_1(d\tau)=\frac{1}{7}|\tau|^2\varphi
  \qquad \pi_7(d\tau)=0
  \qquad
  |\pi_{27}(d\tau)|^2=|d\tau|^2-\frac{1}{7}|\tau|^4
\end{equation}
Consider the functional
\begin{equation}\label{eq:erp-flow-torsion-functional}
  T(t)=\int_M |\tau(t)|^2_{g_{\varphi(t)}}\,
    \textnormal{vol}_{g_{\varphi(t)}}
  =\int_M \frac{\partial\varphi(t)}{\partial t}\wedge\psi(t)
\end{equation}
The second equality follows from the Laplacian flow equation $\partial_t\varphi(t)=d\tau(t)$.
By Bryant's torsion evolution formula phrased in terms of $d\tau$ \cite{Bryant2005}
\begin{equation}\label{eq:erp-flow-torsion-evolution}
  T'(0)=\frac{2}{3}\int_M |\tau|^4\,\textnormal{vol}
    -2\int_M |d\tau|^2\,\textnormal{vol}
\end{equation}
On the other hand, \ref{eq:erp-flow-ansatz} gives $f'(0)=1$ and $f''(0)=\frac{1}{6}|\tau|^2$. The variation of the dual four-form \cite{Bryant2005} is
\begin{equation}\label{eq:erp-flow-dual-variation}
  \left.\frac{\partial\psi(t)}{\partial t}\right|_{t=0}
  =\ast\left(\frac{4}{3}\pi_1(d\tau)
    +\pi_7(d\tau)-\pi_{27}(d\tau)\right)
\end{equation}
Differentiating the second expression in \ref{eq:erp-flow-torsion-functional} and using \ref{eq:erp-flow-projections} gives us
\begin{align}\label{eq:erp-flow-functional-variation}
  T'(0)
  &=\frac{1}{6}\int_M |\tau|^4\,\textnormal{vol}
    +\int_M\left(\frac{4}{3}|\pi_1(d\tau)|^2
      +|\pi_7(d\tau)|^2-|\pi_{27}(d\tau)|^2\right)
      \textnormal{vol}\notag\\
  &=\frac{1}{6}\int_M |\tau|^4\,\textnormal{vol}
    +\int_M\left(\frac{4}{21}|\tau|^4
      -|d\tau|^2+\frac{1}{7}|\tau|^4\right)
      \textnormal{vol}\notag\\
  &=\frac{1}{2}\int_M |\tau|^4\,\textnormal{vol}
    -\int_M |d\tau|^2\,\textnormal{vol}
\end{align}
Comparing \ref{eq:erp-flow-torsion-evolution} with \ref{eq:erp-flow-functional-variation} gives
\begin{equation}\label{eq:erp-flow-integral-equality}
  \int_M |d\tau|^2\,\textnormal{vol}
  =\frac{1}{6}\int_M |\tau|^4\,\textnormal{vol}
\end{equation}
Consequently, $\varphi$ is ERP. \qed

\begin{rmk}
The proof only uses the initial acceleration of the curve. This acceleration cannot be larger than that of the ERP solution and equality holds precisely when the initial structure is ERP.
\end{rmk}

\medskip
Lastly, we characterize the ERP condition by a bound on the $27$-dimensional component of the Weyl curvature.

\begin{prop} [Weyl Curvature Characterization of ERP] \label{thm: ERPWeylCharacterization} Let $\varphi$ be a closed $\GG_2$-Structure with torsion two-form $\tau$ on a $7$-manifold $M$. If $d(\tau\wedge\tau)=0$, then
\begin{equation}\label{eq:weyl-erp-defect}
  \|W_{27}\|^2
  =\frac{9}{320}|\tau|^4
    +\frac{27}{70}|d^{\bar\nabla}\tau|^2
\end{equation}
Consequently, $\varphi$ is ERP if and only if the following pointwise bound holds
\begin{equation}\label{eq:weyl-erp-pinching}
  \|W_{27}\|^2
  \leq\frac{9}{80}\textnormal{Scal}(g_\varphi)^2
\end{equation}
\end{prop}

\textit{Proof.} Cleyton--Ivanov \cite[Equation (6.33)]{CleytonIvanov2008} show that
\begin{equation}\label{eq:canonical-torsion-divergence}
  \left\langle d^{\bar\nabla}\tau,
    \pi_{27}\bigl(\ast(\tau\wedge\tau)\bigr)\right\rangle
  =\frac{1}{3}\ast d(\tau\wedge\tau\wedge\tau)
\end{equation}
Since $d(\tau\wedge\tau)=0$, the identity
\[
  d(\tau\wedge\tau\wedge\tau)
  =\frac{3}{2}\tau\wedge d(\tau\wedge\tau)=0
\]
shows that the two forms in the Cleyton--Ivanov identity are orthogonal. Combining the formulas for $d\tau$, $\textnormal{Ric}^{\ast}_0$, and $d^{\bar\nabla}\tau$ with the definition \ref{eq:weyl-ricci-definition} gives
\begin{equation}\label{eq:weyl-ricci-torsion}
  \i_\varphi(\textnormal{Ric}^{W})
  =-\frac{12}{5}d^{\bar\nabla}\tau
    -\frac{7}{10}\pi_{27}\bigl(\ast(\tau\wedge\tau)\bigr)
\end{equation}
Taking norms in the preceding identity and applying the orthogonality above together with the norm identities from Section 2, we obtain
\[
  |\textnormal{Ric}^{W}|^2
  =\frac{18}{25}|d^{\bar\nabla}\tau|^2
    +\frac{21}{400}|\tau|^4
\]
The definition of $\textnormal{Ric}^{W}$ now gives the stated expression for $\|W_{27}\|^2$. Using the expression for the scalar curvature, the stated bound is equivalent to $\|W_{27}\|^2\leq\frac{9}{320}|\tau|^4$. The expression for $\|W_{27}\|^2$ then forces $d^{\bar\nabla}\tau=0$, which corresponds to the ERP condition. \qed

\begin{rmk}\label{rmk:torsion-properties-not-erp}
The simplicity and harmonicity of $\tau\wedge\tau$, together with constant scalar curvature, do not characterize the ERP condition, even on compact manifolds. Examples follow from the work of Kath and Lauret \cite{KathLauret2020}, where they identify compact ERP $G_2$-Structures; hence an additional hypothesis is required to deduce the ERP condition from these torsion properties. We used the constraint on $W_{27}$ above, but there are likely other equivalent curvature constraints which imply the ERP condition.
\end{rmk}

 We may also use the exceptional pointwise control over the ERP $G_2$-Structure along with a Bochner argument to heavily constrain the automorphism group of these structures. 

\begin{prop} [ERP Automorphism Group] \label{prop: ERPAutomorphism} An ERP $G_2$-Structure $\varphi$ on a compact 7-manifold $M$ does not have any continuous symmetries.

 \end{prop}

 \noindent \textit{Proof.} By definition any vector $X$ in the Lie algebra of the automorphism group satisfies

 \begin{equation}
L_{X}\varphi = 0
 \end{equation}

 Such a vector field is automatically Killing. Moreover, since the Ricci curvature of an ERP $G_2$-Structure is non-positive, from Bochner's theorem any such vector field is parallel and satisfies

 \begin{equation}
\textnormal{Ric}(X,X) = 0
 \end{equation}

 Let the candidate automorphism vector field at a point take the following form in the $G_2$-adapted frame in terms of constant coefficients $\{a_{i}\}$.

 \begin{equation}
X = a_{1}e_{1} + ... + a_{7}e_{7}
 \end{equation}

In terms of these constants the Ricci condition arising from the Bochner formula takes the following form given the pointwise expression for the Ricci tensor in this adapted frame \ref{erppointwise}. Here we without loss of generality choose to label $P_{x} = \langle e_{1}, e_{2}, e_{3}\rangle$ and $Q_{x} = \langle e_{4}, e_{5}, e_{6}, e_{7}\rangle$ in terms of the adapted frame for the $G_2$-Structure at the point \cite{FinoRaffero2021}. 

 \begin{equation}
-\frac{1}{6}|\tau|^2\Big(a_{1}^2 + a_{2}^2 + a_{3}^2 \Big) = 0
 \end{equation}

\noindent As a consequence the vector field must actually pointwise take the form

 \begin{equation}\label{Xpointwise}
X = a_{4}e_{4} + a_{5}e_{5} + a_{6}e_{6} +a_{7}e_{7}
 \end{equation}

In other words, it is supported in the subbundle Q. Following from the expression for the Riemannian metric \ref{G2metric}, for any vector field $X$

 \begin{equation}\label{XNorm}
X\lrcorner\varphi\wedge X\lrcorner\varphi \wedge \varphi = 6|X|^2\textnormal{vol}
 \end{equation}

 Utilizing the fact that the $G_2$-Structure is ERP, we may solve for the three-form $\varphi$ in terms of the torsion two-form in the following manner. 

 \begin{equation}
\varphi = \frac{6}{|\tau|^2}(d\tau - \frac{1}{6}\ast(\tau\wedge\tau))
 \end{equation}

 Note in this expression that since the structure is ERP we have that the scalar curvature is a non-vanishing constant. Plugging this into \ref{XNorm} and multiplying through by the constant $|\tau|^2$ we get

 \begin{equation}
X\lrcorner\varphi\wedge X\lrcorner\varphi\wedge d\tau -\frac{1}{6} X\lrcorner\varphi\wedge X\lrcorner\varphi\wedge\ast(\tau\wedge\tau) = |\tau|^2|X|^2\textnormal{vol}
 \end{equation}

Using the Leibniz property of the interior product on the second term yields

\begin{equation}
X\lrcorner\varphi\wedge X\lrcorner\varphi\wedge d\tau -\frac{1}{6} X\lrcorner\varphi\wedge \varphi\wedge X\lrcorner\ast(\tau\wedge\tau) = |\tau|^2|X|^2\textnormal{vol}
\end{equation}

But since  $\ast(\tau\wedge\tau)$ is supported in $P$, the contraction of $X$ into the form vanishes at every point. With this term vanishing, using the fact that $X$ is a symmetry implies that the left-hand side is exact.

\begin{equation}
d(X\lrcorner\varphi \wedge X\lrcorner\varphi\wedge\tau) = |\tau|^2|X|^2\textnormal{vol}
\end{equation}

Integrating and using compactness implies that $X=0$ \qed

\begin{rmk}
All known examples of ERP $G_2$-Structures on compact 7-manifolds are locally homogeneous and originated from non-compact homogeneous spaces. This result concerning the automorphism group proves that none of the original symmetries present in the non-compact geometry survive the finite quotient. 
\end{rmk}

\section{Natural Perturbations of General ERP $\GG_{2}$-Structures}

We now fix a particular ERP $\GG_2$-Structure. By definition $d\tau$ satisfies the following expression. 

\begin{equation}
d\tau = \frac{1}{6}|\tau|^2\varphi + \frac{1}{6}\ast(\tau\wedge\tau) 
\end{equation}

Noting that the scalar curvature is a non-vanishing constant, we may solve for $\varphi$ in terms of $d\tau$ and $\ast(\tau\wedge\tau)$. 

\begin{equation}
    \varphi = \frac{6}{|\tau|^2} \Big(d\tau - \frac{1}{6}\ast(\tau\wedge\tau) \Big)
\end{equation}

\noindent We can then build new differential forms by tailoring the scale of the differential forms in $\varphi$ 

\begin{align}
    \tilde{\varphi} & = a d\tau - b\ast(\tau\wedge\tau) \\
    & = a \Big(\frac{1}{6}|\tau|^2\varphi + \frac{1}{6}\ast(\tau\wedge\tau)\Big) - b \ast(\tau\wedge\tau) \\
    & \equiv \tilde{a}\varphi - \tilde{b}\ast(\tau\wedge\tau)
\end{align}

Fixing the overall scale, we could view this ansatz as perturbing the original ERP $\GG_2$-Structure by a multiple of $\ast(\tau\wedge\tau)$. We could then compare this with the $\GG_2$-Structures considered by Fino and Raffero \cite{FinoRaffero2021}, which were also perturbations of ERP $\GG_2$-Structures but manifestly within the cohomology class of the form.

\begin{equation}
    \tilde{\varphi} = \varphi + ad\tau 
\end{equation}

Plugging in the expression for $d\tau$ for an ERP $G_2$-Structure, it is apparent that their structures have different relative scales corresponding to the $\varphi$ and $\ast(\tau\wedge\tau)$ terms. In fact, these types of perturbations are more general because we can decouple the relative scales of $\varphi$ and $\ast(\tau\wedge\tau)$ by perturbing in this manner. What is non-trivial about our ansatz is that the three-forms defined in this manner are closed as a consequence of the form $\ast(\tau\wedge\tau)$ being closed for an ERP $\GG_2$-Structure. \\ 

We also need to check when our ansatz is positive, in which case it would define a closed $\GG_2$-Structure. Noting that the tangent bundle splits orthogonally into pieces calibrated by $\tau\wedge\tau$ and its Hodge dual, we may note that in a $\GG_2$-adapted frame at a point \cite{FinoRaffero2021}

\begin{equation}
    \ast(\tau\wedge\tau) = -|\tau|^2 e^{123}
\end{equation}

\noindent In such a frame then at a point the three-form $\tilde{\varphi}$ is given as

\begin{equation}
\tilde{\varphi} = (a+b|\tau|^2)e^{123} + a(- e^{167} - e^{527} - e^{563} + e^{415} +
e^{426} + e^{437})
\end{equation}

As such, the form will be positive if $(a+b|\tau|^2)$ and $a$ have the same sign and in particular don't vanish. Generically then, the ansatz for $\tilde{\varphi}$ will constitute a closed $\GG_2$-Structure. \\

In the context of their perturbations, Fino and Raffero found that $\varphi + a d\tau$ is ERP for all $a$ such that the form is positive. We can generalize this result in the context of our ansatz. 

\begin{prop} [General Perturbations of ERP] \label{prop: GeneralERPPert} For $\varphi$ an ERP $\GG_2$-Structure, a perturbation of the form $\tilde{\varphi} = a\varphi - b\ast(\tau\wedge\tau) $ is ERP whenever it is positive. 
\end{prop} 

\textit{Proof.} We start by finding the torsion $\tilde{\tau}$ associated with $\tilde{\varphi}$. By definition, it satisfies

\begin{equation}
d\tilde{\psi} = -\ast\tilde{\tau}
\end{equation}

To compute $\tilde{\psi}$, we introduce an auxiliary frame defined in terms of the original ERP $\GG_2$-adapted frame by scaling the frame by functions $f$ and $g$.

\begin{align}
    & \tilde{e_{i}} = fe^{i}\hspace{2mm} \textnormal{for} \hspace{2mm} i\in{1,2,3} \\
    & \tilde{e_{i}} = ge^{i} \hspace{2mm}\textnormal{for} \hspace{2mm} i\in{4,5,6,7} 
\end{align}

Here, the prefactors $f$ and $g$ are given in terms of $a$ and $b$ as 

\begin{align}
    & f^3 = a + b|\tau|^2 \\
    & fg^2 = a
\end{align}

We can then compute the Hodge star of $\tilde{\varphi}$ and resolve it into a weighted sum of $\psi$ and $\tau\wedge\tau$ associated to the original ERP $\GG_2$-Structure. 

\begin{equation}
    \tilde{\psi} = f^2g^2\psi + \frac{f^2g^2 - g^2}{|\tau|^2}(\tau\wedge\tau)
\end{equation}

Noting that $\tau\wedge\tau$ is closed we can then see that

\begin{equation}
    d\tilde{\psi} = f^2g^2 d\psi
\end{equation}

This is saying that 

\begin{equation}
    \tilde{\tau}\wedge\tilde{\psi} = f^2g^2 \tau\wedge\psi
\end{equation}

Now we can solve for $\varphi$ associated to the original ERP $\GG_2$-Structure in terms of $\tilde{\varphi}$ and $\ast(\tau\wedge\tau)$.

\begin{equation}
    \varphi = \frac{1}{fg^2} \Big(\tilde{\varphi} + \frac{f^3 - fg^2}{|\tau|^2}\ast(\tau\wedge\tau) \Big)
\end{equation}

We may see that \ref{14dimrepproperties} applied to an ERP $G_2$-Structure since $\tau\wedge\tau\wedge\tau=0$ reduces to

\begin{equation}
\tau\wedge\ast(\tau\wedge\tau) = |\tau|^2\ast\tau
\end{equation}

\noindent Wedging $\varphi$ with $\tau$ then reduces to 

\begin{equation}
\frac{f^2}{g^2} \tilde{\tau}\wedge\tilde{\varphi} = f \tau\wedge\tilde{\varphi}
\end{equation}

Since wedging by any $G_2$-Structure defines an isomorphism between two-forms and five-forms, this implies that $\tilde{\tau}$ is proportional to $\tau$.

\begin{equation}
\tilde{\tau} = \frac{g^2}{f}\tau
\end{equation}

Now we want to check when these closed $\GG_2$-Structures are ERP. To do so, we will also need to compute $|\tilde{\tau}|^2$ and $\ast(\tilde{\tau}\wedge\tilde{\tau})$. The Hodge star and norm are taken with respect to the metric defined by the new $\GG_2$-Structure, not the original ERP one.

\begin{align}
    &\ast(\tilde{\tau}\wedge\tilde{\tau}) = f\ast(\tau\wedge\tau) \\
    &|\tilde{\tau}|^2 = \frac{1}{f^4}|\tau|^2
\end{align}

\noindent Plugging in for each of these quantities we see that since the original data was ERP.

\begin{equation}
    \frac{g^2}{f}d\tau = \frac{g^2}{f} \Big( \frac{1}{6}|\tau|^2\varphi + \frac{1}{6}\ast(\tau\wedge\tau) \Big)
\end{equation}

As such, the perturbed structures are also always ERP regardless as to the values of $f$ and $g$. In addition, both the scalar curvature and the Ricci curvature norm depend only on the parameter $f$ and not on $g$. \qed 

\begin{rmk}
Following from the work of Fino and Raffero, the solution to the Laplacian flow starting at an ERP $G_2$-Structure $\varphi_{0}$ takes the form $\varphi(t) = \varphi_{0} + f(t)d\tau_{0}$ where $f(t)$ is a function exponential in time. As such, the ERP $G_2$-Structures considered here which were up to overall scale of the form $\tilde{\varphi}=\varphi_0 + a\ast_{0}(\tau_{0}\wedge\tau_{0})$ cannot be accessed by the Laplacian flow starting at $\varphi_{0}$ at any time. 
\end{rmk}

\section{Perturbations of Specific Examples}

We now consider three examples from the classification of left-invariant ERP $G_2$-Structures due to Lauret and Nicolini \cite{LauretNicolini2019b}. The types $\mu_B$, $\mu_{M2}$, and $\mu_J$ each admit perturbations obtained by varying two exact summands independently. The type $\mu_B$ is Bryant's example while $\mu_J$ is the example for which Kath and Lauret construct a compact quotient \cite{KathLauret2020}. Throughout this section we use the positive three-form fixed in the classification
\[
\varphi_L=e^{127}+e^{347}+e^{567}+e^{135}-e^{146}-e^{236}-e^{245}.
\]

\begin{prop}[Perturbations of Classified ERP $G_2$-Structures]\label{prop: ClassifiedERPPerturbation}
For the types $\mu_B$, $\mu_{M2}$, and $\mu_J$ in the classification of Lauret and Nicolini, given the $G_2$-Structure $\varphi_L$ above, a perturbation of the form
\[
\widetilde\varphi=a\varphi_L+b\,d(e^{12})+c\,d(e^{56})
\]
for constants $a,b,c$ is ERP whenever it is positive.
\end{prop}

\begin{proof}
We first consider the type $\mu_B$ corresponding to Bryant's example. In the frame used by Lauret and Nicolini its structure equations are \cite{LauretNicolini2019b}
\begin{align*}
de^7&=0, & de^3&=\frac13e^{37}, & de^4&=\frac13e^{47},\\
de^1&=-\frac16e^{17}, & de^2&=-\frac16e^{27},\\
de^5&=\frac13e^{14}+\frac13e^{23}+\frac16e^{57}, &
de^6&=\frac13e^{13}-\frac13e^{24}+\frac16e^{67}.
\end{align*}
The three-form $\varphi_L$ is ERP on $G_B$. The structure equations give
\[
d(e^{12})=\frac13e^{127},\qquad
d(e^{56})=-\frac13(e^{567}+e^{135}-e^{146}-e^{236}-e^{245}).
\]
Thus
\begin{align*}
\widetilde\varphi
&=\overset{A}{a}e^{347}
+\overset{B}{\left(a+\frac b3\right)}e^{127}\\
&\quad+\overset{C}{\left(a-\frac c3\right)}(e^{567}+e^{135}-e^{146}-e^{236}-e^{245}).
\end{align*}
The form is positive in the component containing $\varphi_L$ when $A$, $B$, and $C$ are positive. The other positive components follow after changing signs in the adapted coframe. Define $\widetilde e^i=f_i e^i$ by
\begin{align*}
f_1=f_2&=A^{-1/6}B^{1/3}C^{1/6}, &
f_3=f_4&=A^{1/3}B^{-1/6}C^{1/6},\\
f_5=f_6&=A^{-1/6}B^{-1/6}C^{2/3}, &
f_7&=A^{1/3}B^{1/3}C^{-1/3}.
\end{align*}
In this coframe $\widetilde\varphi$ has the standard pointwise form $\varphi_L$. Its Hodge dual is
\begin{align*}
\widetilde\psi
&=\frac{B^{1/3}C^{5/3}}{A^{2/3}}e^{1256}
+\frac{A^{1/3}C^{5/3}}{B^{2/3}}e^{3456}\\
&\quad+A^{1/3}B^{1/3}C^{2/3}
(e^{1234}+e^{2357}+e^{1457}+e^{1367}-e^{2467}).
\end{align*}
The structure equations give
\[
d\widetilde\psi
=-\frac{A^{1/3}C^{5/3}}{B^{2/3}}e^{34567}
+A^{1/3}B^{1/3}C^{2/3}e^{12347}.
\]
Using $d\widetilde\psi=-\widetilde\ast\widetilde\tau$ gives
\[
\widetilde\tau
=\frac{B^{1/3}C^{2/3}}{A^{2/3}}e^{12}
-\frac{C^{5/3}}{A^{2/3}B^{2/3}}e^{56}.
\]
Another use of the structure equations gives
\begin{align*}
d\widetilde\tau
&=\frac{B^{1/3}C^{2/3}}{3A^{2/3}}e^{127}\\
&\quad+\frac{C^{5/3}}{3A^{2/3}B^{2/3}}
(e^{567}+e^{135}-e^{146}-e^{236}-e^{245}).
\end{align*}
We also have
\begin{align*}
|\widetilde\tau|^2_{\widetilde g}
&=\frac{2C^{2/3}}{A^{2/3}B^{2/3}},\\
\widetilde\ast(\widetilde\tau\wedge\widetilde\tau)
&=-\frac{2A^{1/3}C^{2/3}}{B^{2/3}}e^{347}.
\end{align*}
Substitution gives
\[
d\widetilde\tau
=\frac16|\widetilde\tau|^2_{\widetilde g}\widetilde\varphi
+\frac16\widetilde\ast(\widetilde\tau\wedge\widetilde\tau).
\]
Thus the perturbation is ERP for the type $\mu_B$.

\medskip
We next consider the simply connected Lie group $G_{M2}$. Its Lie algebra has structure equations
\begin{align*}
de^7&=de^3=0, & de^4&=\frac13e^{47},\\
de^1&=-\frac16e^{13}-\frac13e^{17}, &
de^2&=-\frac13e^{14}+\frac16e^{23}-\frac13e^{35},\\
de^5&=\frac13e^{14}+\frac13e^{23}-\frac16e^{35}, &
de^6&=-\frac13e^{24}+\frac16e^{36}-\frac13e^{45}+\frac13e^{67}.
\end{align*}
The three-form $\varphi_L$ is ERP on $G_{M2}$. The structure equations give
\[
d(e^{12})=\frac13(e^{127}+e^{135}),\qquad
d(e^{56})=-\frac13(e^{567}-e^{146}-e^{236}-e^{245}).
\]
Thus
\begin{align*}
\widetilde\varphi
&=\overset{A}{a}e^{347}
+\overset{B}{\left(a+\frac b3\right)}(e^{127}+e^{135})\\
&\quad+\overset{C}{\left(a-\frac c3\right)}(e^{567}-e^{146}-e^{236}-e^{245}).
\end{align*}
The form is positive in the component containing $\varphi_L$ when $A$, $B$, and $C$ are positive. The other positive components follow after changing signs in the adapted coframe. Define $\widetilde e^i=f_i e^i$ by
\begin{align*}
f_1&=A^{-1/6}B^{2/3}C^{-1/6}, &
f_2=f_5&=A^{-1/6}B^{1/6}C^{1/3},\\
f_3=f_7&=A^{1/3}B^{1/6}C^{-1/6}, &
f_4&=A^{1/3}B^{-1/3}C^{1/3},\\
f_6&=A^{-1/6}B^{-1/3}C^{5/6}.
\end{align*}
In this coframe $\widetilde\varphi$ has the standard pointwise form $\varphi_L$. Its Hodge dual is
\begin{align*}
\widetilde\psi
&=\frac{B^{2/3}C^{4/3}}{A^{2/3}}e^{1256}
 +\frac{A^{1/3}C^{4/3}}{B^{1/3}}(e^{3456}-e^{2467})\\
&\quad+A^{1/3}B^{2/3}C^{1/3}(e^{1234}+e^{2357}+e^{1457}+e^{1367}).
\end{align*}
The structure equations give
\[
d\widetilde\psi=-\frac{A^{1/3}C^{4/3}}{B^{1/3}}e^{34567}+A^{1/3}B^{2/3}C^{1/3}e^{12347}.
\]
Using $d\widetilde\psi=-\widetilde\ast\widetilde\tau$ gives
\[
\widetilde\tau=\frac{B^{2/3}C^{1/3}}{A^{2/3}}e^{12}-\frac{C^{4/3}}{A^{2/3}B^{1/3}}e^{56}.
\]
Another use of the structure equations gives
\begin{align*}
d\widetilde\tau
&=\frac{B^{2/3}C^{1/3}}{3A^{2/3}}(e^{127}+e^{135})\\
&\quad+\frac{C^{4/3}}{3A^{2/3}B^{1/3}}(e^{567}-e^{146}-e^{236}-e^{245}).
\end{align*}
We also have
\begin{align*}
|\widetilde\tau|^2_{\widetilde g}&=\frac{2C^{1/3}}{A^{2/3}B^{1/3}},\\
\widetilde\ast(\widetilde\tau\wedge\widetilde\tau)&=-\frac{2A^{1/3}C^{1/3}}{B^{1/3}}e^{347}.
\end{align*}
Substitution gives
\[
d\widetilde\tau=\frac16|\widetilde\tau|^2_{\widetilde g}\widetilde\varphi+\frac16\widetilde\ast(\widetilde\tau\wedge\widetilde\tau).
\]
Thus the perturbation is ERP for the type $\mu_{M2}$.

\medskip
We finally consider the simply connected Lie group $G_J$. Its Lie algebra has structure equations
\begin{align*}
de^7&=de^3=de^4=0,\\
de^1&=\frac{\sqrt2}{6}e^{14}-\frac16e^{17}-\frac{\sqrt2}{6}e^{23}-\frac13e^{36},\\
de^2&=-\frac{\sqrt2}{6}e^{13}-\frac{\sqrt2}{6}e^{24}-\frac16e^{27}+\frac13e^{46},\\
de^5&=\frac12e^{57}, & de^6&=\frac13e^{13}-\frac13e^{24}-\frac16e^{67}.
\end{align*}
The three-form $\varphi_L$ is ERP on $G_J$. The structure equations give
\[
d(e^{12})=\frac13(e^{127}-e^{146}-e^{236}),\qquad d(e^{56})=-\frac13(e^{567}+e^{135}-e^{245}).
\]
Thus
\begin{align*}
\widetilde\varphi
&=\overset{A}{a}e^{347}
+\overset{B}{\left(a+\frac b3\right)}(e^{127}-e^{146}-e^{236})\\
&\quad+\overset{C}{\left(a-\frac c3\right)}(e^{567}+e^{135}-e^{245}).
\end{align*}
The form is positive in the component containing $\varphi_L$ when $A$, $B$, and $C$ are positive. The other positive components follow after changing signs in the adapted coframe. Define $\widetilde e^i=f_i e^i$ by
\begin{align*}
f_1=f_2=f_6&=A^{-1/6}B^{1/2}, & f_3=f_4=f_7&=A^{1/3},\\
f_5&=A^{-1/6}B^{-1/2}C.
\end{align*}
In this coframe $\widetilde\varphi$ has the standard pointwise form $\varphi_L$. Its Hodge dual is
\begin{align*}
\widetilde\psi
&=\frac{BC}{A^{2/3}}e^{1256}+A^{1/3}C(e^{3456}+e^{2357}+e^{1457})\\
&\quad+A^{1/3}B(e^{1234}-e^{2467}+e^{1367}).
\end{align*}
The structure equations give
\[
d\widetilde\psi=-A^{1/3}Ce^{34567}+A^{1/3}Be^{12347}.
\]
Using $d\widetilde\psi=-\widetilde\ast\widetilde\tau$ gives
\[
\widetilde\tau=\frac{B}{A^{2/3}}e^{12}-\frac{C}{A^{2/3}}e^{56}.
\]
Another use of the structure equations gives
\begin{align*}
d\widetilde\tau
&=\frac{B}{3A^{2/3}}(e^{127}-e^{146}-e^{236})\\
&\quad+\frac{C}{3A^{2/3}}(e^{567}+e^{135}-e^{245}).
\end{align*}
We also have
\begin{align*}
|\widetilde\tau|^2_{\widetilde g}&=\frac{2}{A^{2/3}},\\
\widetilde\ast(\widetilde\tau\wedge\widetilde\tau)&=-2A^{1/3}e^{347}.
\end{align*}
Substitution gives
\[
d\widetilde\tau=\frac16|\widetilde\tau|^2_{\widetilde g}\widetilde\varphi+\frac16\widetilde\ast(\widetilde\tau\wedge\widetilde\tau).
\]
Thus the perturbation is ERP for the type $\mu_J$.
\end{proof}

\begin{rmk}
Kath and Lauret construct a lattice $\Gamma\subset G_J$ such that $\Gamma\backslash G_J$ is compact \cite{KathLauret2020}. The left-invariant forms $e^{12}$ and $e^{56}$ descend to this quotient. Therefore every positive perturbation
\[
a\varphi_L+b\,d(e^{12})+c\,d(e^{56})
\]
defines an ERP $G_2$-Structure on the compact quotient. When $a=1$ these perturbations lie in the cohomology class of $\varphi_L$. The choices with $b\neq -c$ are not perturbations by a multiple of $d\tau$ since $\tau=e^{12}-e^{56}$.
\end{rmk}

\medskip
 \subsection*{Acknowledgments} I would like to thank Robert Bryant, Anna Fino, and Jeffrey Streets for helpful comments related to this paper. I would also like to thank Jessica Li for not pinching me too extremely frequently while I studied these Extremally Ricci-Pinched $\GG_2$-Structures. \\

I used ChatGPT for latex formatting, equation referencing, and grammar at various parts of the paper.

\bibliographystyle{plain}
\bibliography{bibLibrary}

\medskip

\noindent Rowland Hall, University of California, Irvine, Irvine, CA 92617

\noindent \textit{Email address:} \texttt{kennona@uci.edu}

\end{document}